\documentclass[centertags,12pt]{amsart}
\usepackage{latexsym}
\usepackage{amsthm}
\usepackage{amssymb}
\usepackage{color}
\usepackage[utf8]{inputenc}
\usepackage{mathrsfs}
\usepackage[english]{babel}
\usepackage{setspace}
\usepackage{textcomp}
\usepackage{graphicx}
\usepackage{cite}
\usepackage{cancel}
\usepackage{comment}

\usepackage{lineno}
\usepackage{hyperref}
\graphicspath{ {images/} }

\numberwithin{equation}{section}
\makeatletter
\renewcommand{\subsection}{\@startsection
{subsection}{2}{0mm}{\baselineskip}{-0.25cm}
{\normalfont\normalsize\bf}}
\makeatother

\newtheorem{theorem}{Theorem}[section]
\newtheorem{proposition}[theorem]{Proposition}
\newtheorem{lemma}[theorem]{Lemma}
\newtheorem{corollary}[theorem]{Corollary}

\newtheorem*{theorem*}{Theorem}

\newtheorem{remark}[theorem]{Remark}
\newtheorem{example}[theorem]{Example}
\newtheorem{result}[theorem]{Result}

\newcommand{\Fq}{\mathbb F_q}
\newcommand{\F}{\mathbb F}
\newcommand{\PG}{\mathrm{PG}}
\newcommand{\N}{\mathrm{N}}
\newcommand{\Tr}{\mathrm{Tr}}
\newcommand{\la}{\langle}
\newcommand{\ra}{\rangle}
\newcommand{\C}{\mathcal C}
\newcommand{\intn}{\mathrm{intn}}

\DeclareMathOperator{\Res}{Res}

\begin{document}

\author[D. Bartoli]{Daniele Bartoli}
\address{ Dipartimento di Matematica e Informatica, Universit\`a degli Studi di Perugia} \email{daniele.bartoli@unipg.it}

\author[C. Zanella]{Corrado Zanella}
\address{ Dipartimento di Tecnica e Gestione dei Sistemi Industriali, Universit\`a degli Studi di Padova} \email{corrado.zanella@unipd.it}

\author[F. Zullo]{Ferdinando Zullo}
\address{ Dipartimento di Matematica e Fisica, Universit\`a degli Studi della Campania ``Luigi Vanvitelli''} \email{ferdinando.zullo@unicampania.it}
\thanks{The third author was partially funded by a fellowship from the Department of Management and Engineering (DTG) of the Padua University and by the project ''VALERE: Vanvitelli pEr la RicErca" of the University of Campania ''Luigi Vanvitelli''.
The research of all the authors was supported by the Italian National Group for Algebraic and Geometric Structures and their Applications (GNSAGA - INdAM)}

\title{A new family of maximum scattered linear sets in $\mathrm{PG}(1,q^6)$}

\begin{abstract}
We generalize the example of linear set presented by the last two authors in ``Vertex properties of maximum scattered linear sets of $\PG(1,q^n)$" (2019) to a more general family, proving that such linear sets are maximum scattered when $q$ is odd and, apart from a special case, they are new.
This solves an open problem posed in ``Vertex properties of maximum scattered linear sets of $\PG(1,q^n)$" (2019).
As a consequence of Sheekey's results in ``A new family of linear maximum rank distance codes" (2016), this family yields to new MRD-codes with parameters $(6,6,q;5)$.
\end{abstract}

\maketitle

\section{Introduction}

Let $\Lambda=\PG(V,\F_{q^n})=\PG(1,q^n)$, where $V$ is a vector space of dimension $2$ over $\F_{q^n}$.
If $U$ is a $k$-dimensional $\F_q$-subspace of $V$, then the \emph{$\F_q$-linear set} $L_U$ is defined as
\[L_U=\{\la {\bf u} \ra_{\mathbb{F}_{q^n}} \colon {\bf u}\in U\setminus \{{\bf 0} \}\},\]
and we say that $L_U$ has \emph{rank} $k$.
Two linear sets $L_U$ and $L_W$ of $\PG(1,q^n)$ are said to be \emph{$\mathrm{P\Gamma L}$-equivalent} if there is an element $\phi$ in $\mathrm{P\Gamma L}(2,q^n)$ such that $L_U^{\phi} = L_W$.
It may happen that two $\F_q$--linear sets $L_U$ and $L_W$ of $\PG(1,q^n)$ are $\mathrm{P\Gamma L}$-equivalent even if the $\F_q$-vector subspaces $U$ and $W$ are not in the same orbit of $\Gamma \mathrm{L}(2,q^n)$ (see \cite{CSZ2015,CMP} for further details).
In this paper we focus on {\it maximum scattered} $\F_q$-linear sets of $\PG(1,q^n)$, that is,
$\F_q$-linear sets of rank $n$ in $\PG(1,q^n)$ of size $(q^n-1)/(q-1)$.

If $\la (0,1) \ra_{\F_{q^n}}$ is not contained in the linear set $L_U$ of rank $n$ of $\PG(1,q^n)$ (which we can always assume after a suitable projectivity), then $U=U_f:=\{(x,f(x))\colon x\in \F_{q^n}\}$ for some linearized polynomial (or \textit{$q$-polynomial}) $f(x)=\sum_{i=0}^{n-1}a_ix^{q^i}\in \F_{q^n}[x]$. In this case we will denote the associated linear set by $L_f$.
If $L_f$ is scattered, then $f(x)$ is called a \emph{scattered} $q$-polynomial; see \cite{Sh}.

The first examples of scattered linear sets were found by Blokhuis and Lavrauw in \cite{BL2000} and by Lunardon and Polverino in \cite{LP2001} (recently generalized by Sheekey in \cite{Sh}).
Apart from these, very few examples are known, see Section \ref{EquivIssue}.

In \cite[Sect.\ 5]{Sh}, Sheekey established a connection between maximum scattered linear sets of $\PG(1,q^n)$ and MRD-codes, which are interesting because of their applications to random linear network coding and cryptography.
We point out his construction in the last section.
By the results of \cite{BartoliMontanucci} and \cite{BartoliZhou}, it seems that examples of maximum scattered linear sets are rare.

In this paper we will prove that any
\begin{equation}\label{ourpol}
f_h (x) = h^{q-1}x^q-h^{q^2-1}x^{q^2}+x^{q^4}+x^{q^5},\quad h\in\F_{q^6},\quad h^{q^3+1}=-1,\quad q\ \mbox{odd}
\end{equation}
is a scattered $q$-polynomial.
This will be done by considering two cases:\\
\underline{Case 1:} $h\in\Fq$, that is, $f_h(x)=x^q-x^{q^2}+x^{q^4}+x^{q^5}$;
the condition $h^{q^3+1}=-1$ implies $q\equiv 1 \pmod 4$.\\
\underline{Case 2:} $h\not\in\Fq$.
In this case $h\neq \pm\sqrt{-1}$, otherwise $h \in \F_{q^2}$ and then we have $h^{q+1}=1$, a contradiction to $h^{q^3+1}=-1$.

Note that in Case 1, this example coincides with the one introduced in \cite{ZZ}, where it has been proved that $f_h$ is scattered for $q\equiv 1 \pmod{4}$ and $q\leq 29$.
In Corollary \ref{EquivLS} we will prove that the linear set $\mathcal L_h$ associated with $f_h(x)$ is new, apart from the case of $q$ a power of 5 and $h\in\F_q$. This solves an open problem posed in \cite{ZZ}.

Finally, in Section \ref{Sec:MRD} we prove that the $\F_q$-linear MRD-codes with parameters $(6,6,q;5)$ arising from linear sets $\mathcal L_h$ are not equivalent to any previously known MRD-code, apart from the case $h \in \F_q$ and $q$ a power of $5$; see Theorem \ref{thm:newMRD}.

\section{$\mathcal{L}_h$ is scattered}

A $q$-\emph{polynomial} (or \emph{linearized polynomial}) over $\F_{q^n}$ is a polynomial of the form
\[f(x)=\sum_{i=0}^t a_i x^{q^i},\]
where $a_i\in \F_{q^n}$ and $t$ is a positive integer.
We will work with linearized polynomials of degree less than or equal to $q^{n-1}$. For such a kind of polynomial, the \emph{Dickson matrix}\footnote{This is sometimes called \emph{autocirculant matrix}.} $M(f)$ is defined as
\[M(f):=
\begin{pmatrix}
a_0 & a_1 & \ldots & a_{n-1} \\
a_{n-1}^{q} & a_0^{q} & \ldots & a_{n-2}^{q} \\
\vdots & \vdots & \vdots & \vdots \\
a_1^{q^{n-1}} & a_2^{q^{n-1}} & \ldots & a_0^{q^{n-1}}
\end{pmatrix} \in \F_{q^n}^{n\times n}
,\]
where $a_i=0$ for $i>t$.

Recently, different results regarding the number of roots of linearized polynomials have been presented, see \cite{qres,teoremone,McGuireSheekey,PZ,Zanella}.
In order to prove that a certain polynomial is scattered, we make use of the following result; see \cite[Corollary 3.5]{qres}.

\begin{theorem}\label{Th:Sottobicchiere}
Consider the $q$-polynomial $f(x)=\sum_{i=0}^{n-1}a_i x^{q^i}$ over $\mathbb{F}_{q^n}$ and, with $m$ as a variable, consider the matrix
\[M(m):=
\begin{pmatrix}
m & a_1 & \ldots & a_{n-1} \\
a_{n-1}^{q} & m^{q} & \ldots & a_{n-2}^{q} \\
\vdots & \vdots & \vdots & \vdots \\
a_1^{q^{n-1}} & a_2^{q^{n-1}} & \ldots & m^{q^{n-1}}
\end{pmatrix}
.\]
The determinant of the $(n-i)\times(n-i)$ matrix obtained by $M(m)$ after removing the first $i$ columns and the last $i$ rows of $M(m)$ is a polynomial $M_{n-i}(m)\in \F_{q^n}[m]$.
Then the polynomial $f(x)$ is scattered if and only if $M_0(m)$ and $M_1(m)$ have no common roots.
\end{theorem}

\subsection{Case 1}\label{sez2}
Let
$$f (x) = x^q - x^{q^2}+ x^{q^4}+ x^{q^5}\in\F_{q^6}[x].$$
By Theorem \ref{Th:Sottobicchiere}, $f(x)$
is scattered if and only if for each $m\in \mathbb{F}_{q^6}$ the determinants
of the following two matrices do not vanish at the same time

\begin{equation}
M_5(m)=\left(
\begin{array}{ccccc}
1&-1&0&1&1\\
m^q&1&-1&0&1\\
1&m^{q^2}&1&-1&0\\
1&1&m^{q^3}&1&-1\\
0&1&1&m^{q^4}&1\\
\end{array}\right),\qquad
M_6(m)=\left(
\begin{array}{cccccc}
m&1&-1&0&1&1\\
1&m^q&1&-1&0&1\\
1&1&m^{q^2}&1&-1&0\\
0&1&1&m^{q^3}&1&-1\\
-1&0&1&1&m^{q^4}&1\\
1&-1&0&1&1&m^{q^5}\\
\end{array}\right).
\end{equation}

\begin{theorem}\label{caso1}
The polynomial $f(x)$ is scattered if and only if $q\equiv 1 \pmod 4$.
\end{theorem}
\begin{proof}
If $q$ is even, then for $m=0$ the matrix $M_6(0)$ has rank two and $f(x)$ is not scattered.

Suppose now $q\equiv 3 \pmod 4$. Then let $\overline{m} \in \mathbb{F}_{q^2}\setminus \mathbb{F}_q$ such that $\overline{m}^2=-4$. So   $\overline{m}=\overline{m}^{q^2}=\overline{m}^{q^4}=-\overline{m}^q=-\overline{m}^{q^3}=-\overline{m}^{q^5}$ and, by direct checking,
$$ \det(M_5(\overline{m}))=(\overline{m}^2+4)^2= 0, \quad\det(M_6(\overline{m}))=-(\overline{m}^2+4)^3= 0$$
and  $f(x)$ is not scattered.

Assume $q\equiv 1 \pmod 4$ and suppose that $f(x)$ is not scattered.
Then there exists $m_0\in \mathbb{F}_{q^6}$ such that
\begin{equation}\label{eq:eqs}
(\det (M_5(m_0)))^{q^s}=0,\ (\det (M_6(m_0)))^ {q^t}=0,\ s,t=0,1,2,3,4,5.
\end{equation}

Consider
\begin{equation}
P_1=\det\left(
\begin{array}{ccccc}
1&-1&0&1&1\\
Y&1&-1&0&1\\
1&Z&1&-1&0\\
1&1&U&1&-1\\
0&1&1&V&1\\
\end{array}\right),\qquad
P_2=\det\left(
\begin{array}{cccccc}
X&1&-1&0&1&1\\
1&Y&1&-1&0&1\\
1&1&Z&1&-1&0\\
0&1&1&U&1&-1\\
-1&0&1&1&V&1\\
1&-1&0&1&1&W\\
\end{array}\right).
\end{equation}

Therefore,
\begin{equation}\label{newroot}
X=m_0,\ Y=m_0^q,\ \ldots,\ W=m_0^{q^5}
\end{equation}
is a root of $P_1=:P_1^{(0)}$, $P_2=:P_2^{(0)}$
and of the polynomials inductively
defined by
\[
P_i^{(j)}(X,Y,Z,U,V,W)=P_i^{(j-1)}(Y,Z,U,V,W,X),\quad j=1,2,3,4,5,\quad i=1,2,
\]
which arise from Equation \ref{eq:eqs}. These polynomials satisfy
$$\left(P_i^{(j-1)}(m_0,m_0^q,m_0^{q^2},m_0^{q^3},m_0^{q^4},m^{q^5})\right)^q=P_i^{(j)}(m_0,m_0^q,m_0^{q^2},m_0^{q^3},m_0^{q^4},m^{q^5}).$$
One obtains a set $S$ of twelve equations in $X,Y,Z,U,V,W$
having a nonempty zero set.
The following arguments are based on the fact that taking the resultant $R$ of two
polynomials in $S$ with respect to any variable, the equations
$S\cup\{R\}$ admit the same solutions.

We have
\begin{equation}\label{P1}
P_1=YZUV - YZU - 2YZ + 2YU + 4Y - ZUV + 2ZV - 2UV + 4V + 16=0.
\end{equation}
Consider the following resultants:

\begin{eqnarray*}
Q_1&:=&
\Res_V(P_1^{(3)},P_1)=2(XY^2ZU - XY^2ZW + X Y^2 U W + 2 X Y^2 W - 2 X Y Z U  \\
           &&+ 2 X Y Z W- 2 X Y U W + 8 X Y W + 8 X Y - 8 X W + 16 X - Y^2 Z U W - 2 Y^2 Z U \\
            &&+2 Y Z U W - 8 Y Z U - 8 Y Z + 8 Y U - 8 Y W + 8 Z U - 16 Z + 16 U -
            16 W),\\
Q_2&:=&
\Res_V(P_1^{(4)},P_1)=XYZW - XYZ - XYW + 2XZ  \\
&&- 2XW- 2YZ + 2YW + 4Z + 4W + 16,\\
Q_3&:=&
\Res_V(P_1^{(5)},P_1)=XYZU - XYZ - 2XY + 2XZ\\
&& + 4X - YZU + 2YU - 2ZU + 4U + 16.
\end{eqnarray*}
They all must be zero, as well as
\begin{equation}
\Res_W(\Res_U(Q_1,Q_3),Q_2)=8(YZ - 4)(Y^2 + 4)(X - Z)(XZ + 4)(XY - 4).
\end{equation}
We distinguish a number of cases.
\begin{enumerate}
\item Suppose that $Y^2=-4$. Since $q\equiv 1 \pmod 4$, $X=Y=Z=U=V=W$. So
$$P_1=X^4 - 2X^3 + 8X + 16$$
and the resultant between $X^2+4$ and $P_1$ with respect to $X$ is $2^{27}\neq 0$ and then
(\ref{newroot}) is not a root of $P_1$, a contradiction. 
\item Condition $YZ=4$ is clearly equivalent to $XY=4$. This means that $Y=U=W=4/X$, $Z=V=X$. Therefore, by \eqref{P1}
we get $X^2+4=0$ and we proceed as above.
\item Case $XZ=-4$. In this case $Z=-4/X$, $U=-4/Y$, $V=-4/Z=X$, $W=Y$, $X=Z$ and therefore $X^2=-4$ and we can proceed as above.
\item Condition $X=Z$ implies $X\in \mathbb{F}_{q^2}$ and so $X=Z=V$ and $Y=U=W$.
By substituting in $P_1$ and $P_2$,
\begin{eqnarray*}
X^3Y^3 + 3X^3Y - 6X^2Y^2 - 12X^2 + 3XY^3 + 24XY - 12Y^2 - 64&=0,\\
X^2Y^2 - X^2Y + 2X^2 - XY^2 - 4XY + 4X + 2Y^2 + 4Y + 16&=0.
\end{eqnarray*}
Eliminating $Y$ from these  two equations one gets
$$8(X^2 + 4)^6=0,$$
and so  $X^2+4=0$. We proceed as in the previous cases.
\end{enumerate}
This proves that such $m_0\in \mathbb{F}_{q^6}$ does not exist and the assertion follows.
\end{proof}

\subsection{Case 2}

We apply the same methods as in Section \ref{sez2}.
In the following preparatory lemmas (and in the rest of the paper) $q$ is a power of an arbitrary prime $p$.
\begin{lemma}\label{Lemma1}
Let $h \in \mathbb{F}_{q^6}$ be such that $h^{q^3+1}=-1$, $h^4 \neq 1$.
Then
\begin{enumerate}
\item $h^q\neq -h$;
\item $h^{q^2+1}\neq 1$;
\item $h^{q^2+1}\neq \pm h^q$, if $q$ is odd;
\item $h^{4q^2+4}+14h^{2q^2+2q+2}+h^{4q}=0$ implies $p=2$ and $h^{q^2-q+1}=1$ or $q=3^{2s}$, $s\in\mathbb N^*$, $h^{q^2-q+1}=\pm \sqrt{-1}$.
\end{enumerate}
\end{lemma}
\proof
The first three are easy computations. Consider now
$$h^{4q^2+4}+14h^{2q^2+2q+2}+h^{4q}=0.$$
For $p=2$ the equation above implies $h^{q^2-q+1}=1$.

Assume now $p\neq2$.
Since $h\neq 0$, it is equivalent to
$$(h^{q^2-q+1})^4+14(h^{q^2-q+1})^2+1=0,$$
that is $(h^{q^2-q+1})^2=-7\pm 4\sqrt{3}=(\sqrt{-3}\pm2\sqrt{-1})^2$. Let $z=-7\pm 4\sqrt{3}$. Note that  $h^{q^2-q+1}=\pm \sqrt{z}$ belongs to $\mathbb{F}_{q^2}$. We distinguish two cases.
\begin{itemize}
\item $\sqrt{z}\in \mathbb{F}_q$. Then
$$-1=h^{q^3+1}=(h^{q^2-q+1})^{q+1}=\left(\pm \sqrt{z}\right)^{q+1}=z=-7\pm4\sqrt{3},$$
a contradiction if $p\neq 3$.
Also, $z=-1$, $q$ is an even power of $3$, and $h^{q^2-q+1}=\pm \sqrt{-1}$.
\item $\sqrt{z}\notin \mathbb{F}_q$. Then
$$-1=h^{q^3+1}=(h^{q^2-q+1})^{q+1}=\left(\pm \sqrt{z}\right)^{q+1}=-z=7\mp4\sqrt{3},$$
a contradiction if $p\neq2$.
\end{itemize}

\endproof

\begin{lemma}\label{Lemma2}
Let $h \in \mathbb{F}_{q^6}$ be such that $h^{q^3+1}=-1$, $h^4\neq 1$.
If a root $\sigma$ of the polynomial
$$h^{q+1}T^{q+1} + (h^{q^2+q+2} + h^{2q^2+2})T^q + (h^{2q^2+2} - h^{q^2+1})T+ h^{q^2+2q+1} + h^{2q^2+q+1}
 - h^{2q} - h^{q^2+q}\in \mathbb{F}_{q^6}[T]$$
belongs to $\mathbb{F}_{q^6}$, then one of the following cases occurs:
\begin{itemize}
\item $p=2$, $h^{q^2-q+1}=1$; or
\item $q=3^{2s}$, $s>0$, $h^{q^2-q+1}=\pm \sqrt{-1}$; or
\item $\sigma=\pm(h^{q^2}+h^q)$; or
\item $h \in \mathbb{F}_q$.
\end{itemize}
\end{lemma}
\proof
First, note that $\sigma=0$ would imply $h^q(h^{q} + h)^q(h^{q^2+1}- 1)=0$ which is impossible by Lemma \ref{Lemma1}. Therefore $\sigma\neq 0$ and $\sigma^{q^i}=\frac{\ell_i(X)}{m_i(X)}$, where
{\footnotesize
\begin{eqnarray*}
\ell_1(X)&=& -(h^{q^2+1} - 1)(h^{q^2+1}X + h^{2q} + h^{q^2+q})\\
m_1(X)&=&h(h^qX + h^{q^2+q+1} + h^{2q^2+1})\\
\ell_2(X)&=&-(h^q + h)(2 h^{q^2+q+1}X + h^{2q^2+q+2} + h^{3q^2+2} + h^{3q} + h^{q^2+2q})\\
m_2(X)&=&h^{q+1}(h^{2q^2+2}X + h^{2q}X + 2h^{q^2+2q+1} + 2h^{2q^2+q+1})\\
\ell_3(X)&=&(h^q + h)^q(3 h^{2q^2+q+2}X + h^{3q}X + h^{3q^2+q+3} + h^{4q^2+3} + 3h^{q^2+3q+1}+ 3h^{2q^2+2q+1})\\
m_3(X)&=&h^{q^2+q}(h^{3q^2+3}X + 3h^{q^2+2q+1} X + 3h^{2q^2+2q+2} + 3 h^{3q^2+q+2} + h^{4q} + h^{q^2+3q})\\
\ell_4(X)&=&(h^{q^2+1} - 1)(h^{4q^2+4}X + 6 h^{2q^2+2q+2}X  + h^{4q}X + 4 h^{3q^2+2q+3} + 4 h^{4q^2+q+3} + 4 h^{q^2+4q+1} + 4 h^{2q^2+3q+1})\\
m_4(X)&=&h^{q^2}(4 h^{3q^2+q+3}X + 4h^{q^2+3q+1} X + h^{4q^2+q+4} + h^{5q^2+4} + 6 h^{2q^2+3q+2} + 6 h^{3q^2+2q+2} + h^{5q} + h^{q^2+4q})\\
\ell_5(X)&=&-(h^q + h)(h^{5q^2+5}X + 10h^{3q^2+2q+3} X + 5h^{q^2+4q+1} X + 5h^{4q^2+2q+4}  + 5 h^{5q^2+q+4} + 10h^{2q^2+4q+2}\\
		&& + 10h^{3q^2+3q+2} + h^{6q} + h^{q^2+5q})\\
m_5(X)&=&5h^{4q^2+q+4} X  + 10 h^{2q^2+3q+2}X + h^{5q}X + h^{5q^2+q+5} + h^{6q^2+5} + 10 h^{3q^2+3q+3} + 10h^{4q^2+2q+3} \\
		&&+ 5 h^{q^2+5q+1} + 5h^{2q^2+4q+1}\\
\ell_6(X)&=&(h^q + h)^q(6 h^{5q^2+q+5}X + 20h^{q^3+3q+3} X + 6 Xh^{q^2+5q+1} + h^{6q^2+q+6}+ h^{7q^2+6} + 15 h^{4q^2+3q+4} \\
		&&+ 15 h^{5q^2+2q+4} + 15 h^{2q^2+5q+2} + 15 h^{3q^2+4q+2} +  h^{7q} + h^{q^2+6q})\\
m_6(X)&=&h^{6q^2+6}X+ 15h^{4q^2+2q+4} X + 15 h^{2q^2+4q+2}X + h^{q^6}X + 6h^{5q^2+2q+5} + 6 h^{6q^2+q+5} + 20h^{3q^2+4q+3}\\
		&&+ 20 h^{4q^2+3q+3} + 6 h^{q^2+6q+1} + 6 h^{2q^2+5q+1}.\\
\end{eqnarray*}}

Since $\sigma^{q^6}=\sigma$, in particular
$$(h^{2q^2+2}+h^{2q})(h^{4q^2+4}+14 h^{2q^2+2q+2}+h^{4q})( h^{q^2}-h^q)(\sigma + h^q + h^{q^2})(\sigma - h^q - h^{q^2})=0.$$
The claim follows from Lemma \ref{Lemma1}.
\endproof

\begin{lemma}\label{Lemma3}
Let $h \in \mathbb{F}_{q^6}$ be such that $h^{q^3+1}=-1$, $h^4=1$.
If a root $\sigma$ of the polynomial
$$h^{q+1}T^{q^2+1} + (h^q + h)^{q+1}\in \mathbb{F}_{q^6}[T]$$
belongs to $\mathbb{F}_{q^6}$, then
$$\sigma=\pm(h^{q^2}+h^q).$$
\end{lemma}
\proof
If $\sigma=0$, then $h^q + h=0$, a contradiction to Lemma \ref{Lemma1}. So we can suppose $\sigma\neq 0$. Then
\begin{eqnarray*}
\sigma^{q^2}&=& -\frac{(h^{q-1}+1)^{q+1}}{\sigma}\\
\sigma^{q^4}&=& (h^{q-1}+1)^{q^3+q^2-q-1}\sigma\\
\sigma^{q^6}&=& -\frac{(h^{q-1}+1)^{q^5+q^4-q^3-q^2+q+1}}{\sigma}=\frac{(h^{q}+h)^{2q}}{\sigma}.\\
\end{eqnarray*}
So, $\sigma=\pm(h^{q^2}+h^{q})$.
\endproof

%
%

Let $h\in \mathbb{F}_{q^6}$ be such that $h^{q^3+1}=-1$, $h^4\neq 1$. By Theorem \ref{Th:Sottobicchiere} the polynomial
$$f_h (x) = h^{q-1}x^q-(h^{q^2-1})x^{q^2}+x^{q^4}+x^{q^5}$$
is scattered if and only if for each $m\in \mathbb{F}_{q^6}$ the determinant of the following two matrices do not vanish at the same time

\begin{equation}
M_6(m)=\left(
\begin{array}{cccccc}
m &			h^{q-1} &			-h^{q^2-1} &		0 &			1 &				1\\
1 &			m^q	&			h^{q^2-q} &		h^{-q-1} &		0 &				1\\
1 &			1 &				m^{q^2} &			-h^{-q^2-1} &	h^{-q^2-q} &		0\\
0 &			1 &				1 &				m^{q^3} &		h^{1-q} &			-h^{1-q^2}\\	
h^{q+1} &		0 &				1 &				1 &			m^{q^4}&			h^{q-q^2}\\	
-h^{q^2+1} &	h^{q^2+q} &		0 &				1 &			1 &				m^{q^5}\\
\end{array}\right),
\end{equation}
\begin{equation}
M_5(m)=\left(
\begin{array}{ccccc}
h^{q-1} &			-h^{q^2-1} &		0 &			1 &				1\\
m^q	&			h^{q^2-q} &		h^{-q-1} &		0 &				1\\
1 &				m^{q^2} &			-h^{-q^2-1} &	h^{-q^2-q} &		0\\
1 &				1 &				m^{q^3} &		h^{1-q} &			-h^{1-q^2}\\	
0 &				1 &				1 &			m^{q^4}&			h^{q-q^2}\\	
\end{array}\right).
\end{equation}

\begin{theorem}\label{Th:p=2}
Let $h\in \mathbb{F}_{q^6}$, $q=2^s$, be such that $h^{q^3+1}=1$.
Then the polynomial $f_h (x) = h^{q-1}x^q-(h^{q^2-1})x^{q^2}+x^{q^4}+x^{q^5}$ is not scattered.
\end{theorem}
\proof
Consider $\overline{m}=h^{q^2}+h^q$. So,
$$\overline{m}^{q}=1/h+h^{q^2}, \quad \overline{m}^{q^2}=1/h^q+1/h, \quad \overline{m}^{q^3}=1/h^{q^2}+1/h^{q}, \quad \overline{m}^{q^4}=h+1/h^{q^2}, \quad \overline{m}^{q^5}=h^q+h.$$
By direct checking, in this case, both $\det(M_6(\overline{m}))=\det(M_5(\overline{m}))=0$ and therefore $f_h(x)$ is not scattered.
\endproof

\begin{theorem}\label{Th:GeneralCase}
Let $h\in \mathbb{F}_{q^6}$, $q=p^s$, $p>2$, be such that $h^{q^3+1}=-1$ and $h \notin \mathbb{F}_q$.
Then the polynomial $f_h (x) = h^{q-1}x^q-(h^{q^2-1})x^{q^2}+x^{q^4}+x^{q^5}$ is scattered.
\end{theorem}
\begin{proof}
First we note that  $h^4\neq 1$ since $q$ is odd, $h\notin \mathbb{F}_q$, and $h^{q^3+1}=-1$.
Suppose that $f(x)$ is not scattered. Then $\det(M_6(m_0))=\det(M_5(m_0))=0$ for some $m_0\in\F_{q^6}$.
Consider
$$X=m_0, \quad Y=m_0^q, \quad Z=m_0^{q^2}, \quad U=m_0^{q^3}, \quad V= m_0^{q^4}, \quad W = m_0^{q^5}.$$
With a procedure similar to the one in the proof of Theorem \ref{caso1}, we will compute resultants starting from the polynomials
associated with $\det(M_6(m_0))$, $\det(M_5(m_0))^{q^3}$, and $\det(M_5(m_0))^{q^5}$.
Eliminating $W$ using $\det(M_5(m_0))^{q^3}=0$ and $U$ using $\det(M_5(m_0))^{q^5}=0$, one gets from $\det(M_6(m_0))=0$
$$h^{q^2+2q+1}\varphi_1(X,Y)\varphi_2(X,Y,Z,V)\varphi_3(X,Y,Z,V)=0,$$
where
{\scriptsize
\begin{eqnarray*}
\varphi_1(X,Y)&=&h^{q+1}XY + h^{2q^2+2}X -h^{q^2+1} X + h^{q^2+q+2}Y + h^{2q^2+2}Y + h^{q^2+2q+1} + h^{2q^2+q+1} - h^{2q}- h^{q^2+q};\\
\varphi_2(X,Y,Z,V)&=&h^{q^2+q+2}X Y Z V- h^{q^2+q+2}X Y Z - h^2X Y - h^{q+1}X Y - h^{2q^2+q+1}X Z V- h^{2q^2+2}X V - h^{2q^2+q+1}X V \\
				&&-h^{q^2+2q+3}Y Z - h^{2q^2+q+3}Y Z - h^{q^2+q+2}Y - h^{2q^2+2}Y - h^{q^2+2q+1}Y - h^{2q^2+q+1}Y - h^{q^2+2q+1}Z V\\
				&&- h^{2q^2+q+1}Z V - h^{2q^2+q+1}V  - h^{3q^2+1}V - h^{2q^2+2q}V - h^{3q^2+q}V  + h^{2q^2+q+3} + h^{3q^2+3} + h^{2q^2+2q+2}\\
				&&+h^{3q^2+q+2} - 2h^{q^2+q+2} - 2h^{2q^2+2} - 2 h^{q^2+2q+1}- 2h^{2q^2+q+1} + h^{q+1}+ h^{q^2+1}+ h^{2q} + h^{q^2+q};\\
\varphi_3(X,Y,Z,V)&=&h^{q^2+q+2}X Y Z V+ h^{q^2+q+2}X Y Z - h^2X Y - h^{q+1}X Y + h^{2q^2+q+1}X Z V- h^{2q^2+2}X V - h^{2q^2+q+1}X V \\
				&&-h^{q^2+2q+3}Y Z - h^{2q^2+q+3}Y Z + h^{q^2+q+2}Y + h^{2q^2+2}Y + h^{q^2+2q+1}Y + h^{2q^2+q+1}Y - h^{q^2+2q+1}Z V\\
				&&- h^{2q^2+q+1}Z V + h^{2q^2+q+1}V  + h^{3q^2+1}V + h^{2q^2+2q}V + h^{3q^2+q}V  + h^{2q^2+q+3} + h^{3q^2+3} + h^{2q^2+2q+2}\\
				&&+h^{3q^2+q+2} - 2h^{q^2+q+2} - 2h^{2q^2+2} - 2 h^{q^2+2q+1}- 2h^{2q^2+q+1} + h^{q+1}+ h^{q^2+1}+ h^{2q} + h^{q^2+q}.\\	
\end{eqnarray*}}
\begin{itemize}
\item If $\varphi_1(X,Y)=0$, then by Lemma \ref{Lemma2} either $q=3^{2s}$ and $h^{q^2-q+1}=\pm \sqrt{-1}$, or $X=\pm(h^{q^2}+h^q)$.

In this last case,
\begin{eqnarray}\label{Valori}
Y=\pm(-h^{-1}+h^{q^2}), \qquad Z=\pm(-h^{-q}-h^{-1}), \qquad U=\pm(-h^{-q^2}-h^{-q})\nonumber\\
V=\pm(h-h^{-q^2}), \qquad W=\pm(h^q+h).
\end{eqnarray}
By substituting in $\det(M_5(m_0))$ one obtains
$$4(h+h^q)^{q+1}(h^{q^2+1}-1)(h^{q^2+1}-h^q)=0$$
and
$$4(h+h^q)^{q+1}(h^{q^2+1}-1)(h^{q^2+1}+h^q)=0,$$
respectively. Both are not possible due to Lemma \ref{Lemma1}.

Consider now the case $q=3^{2s}$, $h^{q^2-q+1}=\pm \sqrt{-1}$ and $X\neq \pm (h^{q^2}+h^q)$. So, using $\varphi_1(X,Y)=0$ and $h^{q^2-q+1}=\pm \sqrt{-1}$,
$$\det(M_5(m_0))=0 \Longrightarrow h^{q^2+2q+1}(h^{q^2}+h^q)(h^q+h)(h^{q^2+1}-1)(h^{q^2+q}+h^q)^3(h^{q^2+q}-h^q)^3\cdot$$
$$ \cdot (h^{2q^2+2}-h^{q^2+1}+h^{2q})(X+h^q+h^{q^2})^2(X-h^q-h^{q^2})^2=0. $$
By Lemma \ref{Lemma1} we get
\[ h^{2q^2+2}-h^{q^2+1}+h^{2q}=0, \]
which yields to a contradiction.

\item If $\varphi_2(X,Y,Z,V)=0$ and $\varphi_1(X,Y)\neq 0$, eliminating $V$ in $\det(M_5(m_0))=0$ one gets
\begin{eqnarray*}
2h^{3q^2+2q+1}
(h^{q+2}Y Z - h^{q^2+2} - h^{q^2+q+1} + h^q + h)\cdot\\
(hXY + h^{q^2+q+1} + h^{2q^2+1} - h^{q^2} - h^q)\cdot\\
(h^{q+1}XZ + h^{q+1} + h^{q^2+1}+ h^{2q} + h^{q^2+q})\cdot\\
(h^{q+2}YZ + hY + h^qY - h^{q^2+q+1}Z + h^qZ - h^{q^2+2} - h^{q^2+q+1} + h^q + h)
=0.
\end{eqnarray*}
\begin{itemize}
\item If $h^{q+2}Y Z - h^{q^2+2} - h^{q^2+q+1} + h^q + h=0$ then, from
$$Z=\frac{h^{q^2+2} + h^{q^2+q+1} - h^q - h}{h^{q+2}Y},$$
$\det(M_5)=0$ gives
$$(h^q+h)^{q+1}(hY - h^{q^2+1}+ 1)(hY + h^{q^2+1}- 1)=0.$$
So, \eqref{Valori} holds and as in the case $\varphi_1(X,Y)=0$ a contradiction arises.
\item If $hXY + h^{q^2+q+1} + h^{2q^2+1} - h^{q^2} - h^q=0$ then, from
$$Y=\frac{-h^{q^2+q+1} - h^{2q^2+1} + h^{q^2} + h^q}{hX},$$
the equation $\det(M_5(m_0))=0$ yields
$$(h^q+h)(h^{q^2+1}-1)(X-h^{q^2}-h^q)(X+h^{q^2}+h^q)=0.$$
So, \eqref{Valori} holds and as in the case $\varphi_1(X,Y)=0$, a contradiction.
\item If $h^{q+1}XZ + h^{q+1} + h^{q^2+1}+ h^{2q} + h^{q^2+q}=0$ then by Lemma \ref{Lemma3}
$$(X-h^{q^2}-h^q)(X+h^{q^2}+h^q)=0,$$
again a contradiction as before.
\item If $h^{q+2}YZ + hY + h^qY - h^{q^2+q+1}Z + h^qZ - h^{q^2+2} - h^{q^2+q+1} + h^q + h=0$ then
$$Z = -\frac{(h^q+h)Y - h^{q^2+2} - h^{q^2+q+1} + h^q + h}{ h^{q+2}Y - h^{q^2+q+1} + h^q}.$$
So, substituting  $U=Z^q$, $V=Z^{q^2}$, $W=Z^{q^3}$, $X=Z^{q^4}$ in $\det(M_5(m_0))=0$ we get
$$(h - 1)^{q+1}(h+1)^{q+1}(h^q+h)^{q+1}(h^{q^2+1}-1)(hY-h^{q^2+1}+1)^2(hY+h^{q^2+1}-1)^2=0.$$
By Lemma \ref{Lemma1},  $(hY-h^{q^2+1}+1)(hY+h^{q^2+1}-1)=0$. Since $Y=\pm(h^{q^2}-1/h)$ then  \eqref{Valori} holds and a contradiction arises as in the case $\varphi_1(X,Y)=0$.
\end{itemize}
\item If $\varphi_3(X,Y,Z,V)=0$ and $\varphi_1(X,Y)\neq 0$, eliminating $U$ from $\det(M_5(m_0))=0=\det(M_5(m_0))^{q^5}$ and then eliminating $V$ using $\varphi_3(X,Y,Z,V)=0$ one gets
\begin{eqnarray*}
2h^{3q^2+q+1}(h^q+h)^q
(h^{q+2}Y Z - h^{q^2+2} - h^{q^2+q+1} + h^q + h)^2\cdot\\
(hXY + h^{q^2+q+1} + h^{2q^2+1} - h^{q^2} - h^q)\cdot\\
(h^{q+1}XZ + h^{q+1} + h^{q^2+1}+ h^{2q} + h^{q^2+q})=0.
\end{eqnarray*}
A contradiction follows as in the case $\varphi_2(X,Y,Z,V)=0$ and $\varphi_1(X,Y)\neq 0$.
\end{itemize}
\end{proof}

\section{The equivalence issue}\label{EquivIssue}

We will deal with the linear sets $\mathcal L_h=L_{f_h}$ associated with the polynomials defined in (\ref{ourpol}).
Note that when $h \in \F_q$, such a linear set coincide with the one introduced in \cite[Section 5]{ZZ}.

\subsection{Preliminary results}

We start by listing the non-equivalent (under the action of $\Gamma\mathrm{L}(2,q^6)$) maximum scattered subspaces of $\F_{q^6}^2$, i.e.\ subspaces defining maximum scattered linear sets.

\begin{example}\label{exKnownscattered}
\begin{enumerate}
\item $U^{1}:= \{(x,x^{q}) \colon x\in \F_{q^6}\}$, defining the linear set of pseudoregulus type, see \cite{BL2000,CsZ20162};
\item $U^{2}_{\delta}:= \{(x,\delta x^{q} + x^{q^{5}})\colon x\in \F_{q^6}\}$, $\N_{q^6/q}(\delta)\notin \{0,1\}$, defining the linear set of LP-type, see \cite{LP2001,LMPT2015,LTZ,Sh};
\item $U^{3}_{\delta}:= \{(x, x^{q}+\delta x^{q^{4}})\colon x\in \F_{q^{6}}\}$, $\N_{q^6/q^{3}}(\delta) \notin \{0,1\}$, satisfying further conditions on $\delta$ and $q$, see \cite[Theorems 7.1 and 7.2]{CMPZ} and \cite{PZ} \footnote{Here $q>2$, otherwise it is not scattered.};
\item $U^{4}_{\delta}:=\{(x, x^q+x^{q^3}+\delta x^{q^5}) \colon x \in \F_{q^6}\}$, $q$ odd and $\delta^2+\delta=1$, see \cite{CsMZ2018,MMZ}.
\end{enumerate}
\end{example}

In order to simplify the notation, we will denote by $L^1$ and $L^{i}_{\delta}$ the $\F_q$-linear set defined by $U^{1}$ and $U^{i}_{\delta}$, respectively.
We will also use the following notation: $\mathcal{U}_h:=U_{h^{q-1}x^q-h^{q^2-1}x^{q^2}+x^{q^4}+x^{q^5}}$.

\begin{remark}
Consider the non-degenerate symmetric bilinear form of $\F_{q^6}$ over $\F_q$ defined by
\[ \la x,y\ra= \Tr_{q^6/q}(xy), \]
\noindent for each $x,y \in \F_{q^6}$. Then the \emph{adjoint} $\hat{f}$ of the linearized polynomial $\displaystyle f(x)=\sum_{i=0}^{5} a_ix^{q^i} \in \tilde{\mathcal{L}}_{6,q}$ with respect to the bilinear form $\la,\ra$ is
\[ \hat{f}(x)=\sum_{i=0}^{5} a_i^{q^{6-i}}x^{q^{6-i}}, \]
i.e.
\[ \Tr_{q^6/q}(xf(y))=\Tr_{q^6/q}(y\hat{f}(x)), \]
for any $x,y \in \F_{q^6}$.
\end{remark}

In \cite[Propositions 3.1, 4.1 \& 5.5]{CsMZ2018} the following result has been proved.

\begin{lemma}\label{equiv}
Let $L_f$ be one of the maximum scattered of $\PG(1,q^6)$ listed before.
Then a linear set $L_U$ of $\PG(1,q^6)$ is $\mathrm{P}\Gamma\mathrm{L}$-equivalent to $L_f$ if and only if $U$ is $\Gamma\mathrm{L}$-equivalent either to $U_f$ or to $U_{\hat{f}}$
Furthermore,  $L_U$ is $\mathrm{P}\Gamma\mathrm{L}$-equivalent to
$L^3_{\delta}$ if and only if $U$ is $\Gamma\mathrm{L}$-equivalent to $U_\delta^3$.
\end{lemma}

\smallskip

We will work in the following framework.
Let $x_0,\ldots,x_{5}$ be the homogeneous coordinates of $\PG(5,q^6)$ and let
\[ \Sigma=\{\la (x,x^{q},\ldots,x^{q^{5}}) \ra_{\F_{q^6}} \colon x \in\F_{q^6} \} \]
be a fixed canonical subgeometry of $\PG(5,q^6)$.
The collineation $\hat{\sigma}$ of $\PG(5,q^6)$ defined by
$\la(x_0,\ldots,x_{5})\ra_{\F_{q^6}}^{\hat{\sigma}}=\la(x_{5}^{q},x_0^{q},\ldots,x_{4}^{q})\ra_{\F_{q^6}}$ fixes precisely the points of $\Sigma$.
Note that if $\sigma$ is a collineation of $\PG(5,q^6)$ such that $\mathrm{Fix}(\sigma)=\Sigma$, then $\sigma=\hat{\sigma}^s$, with $s \in \{1,5\}$.

Let $\Gamma$ be a subspace of $\PG(5,q^6)$ of dimension $k\ge0$
such that $\Gamma \cap \Sigma=\emptyset$, and $\dim(\Gamma\cap\Gamma^\sigma)\ge k-2$.
Let $r$ be the least positive integer satisfying the condition
\begin{equation}\label{def_r}
\dim(\Gamma\cap\Gamma^\sigma\cap\Gamma^{\sigma^2}\cap\ldots\cap\Gamma^{\sigma^r})>k-2r.
\end{equation}
Then we will call the integer $r$ the \emph{intersection number of} $\Gamma$ w.r.t.\ $\sigma$ and we will denote it by $\mathrm{intn}_{\sigma}(\Gamma)$; see \cite{ZZ}.

Note that if $\hat\sigma$ is as above, then $\intn_{\hat\sigma}(\Gamma)=\intn_{\hat\sigma^5}(\Gamma)$ for any $\Gamma$.

As a consequence of the results of \cite{CsZ20162,ZZ} we have the following result.

\begin{result}\label{nopseudonoLP}
Let $L$ be a scattered linear set of $\Lambda=\PG(1,q^6)$ which can be realized in $\PG(5,q^6)$ as the projection of $\Sigma=\mathrm{Fix}(\sigma)$ from $\Gamma\simeq\PG(3,q^6)$ over $\Lambda$.
If $\mathrm{intn}_{\sigma}(\Gamma)\neq 1,2$, then $L$ is not equivalent to any linear set neither of pseudoregulus type nor of LP-type.
\end{result}

\subsection{$\mathcal{L}_h$ is new in most of the cases}

The linear set $\mathcal{L}_h$ can be obtained by projecting the canonical subgeometry $\Sigma=\{ \langle(x,x^q,x^{q^2},x^{q^3},x^{q^4},x^{q^5})\rangle_{\F_{q^6}} \colon x \in \F_{q^6}^* \}$ from
\[ \Gamma \colon \left\{ \begin{array}{ll} x_0=0 \\ h^{q-1}x_1-h^{q^2-1}x_2+x_4+x_5=0 \end{array} \right. \]
to
\[ \Lambda \colon \left\{ \begin{array}{llll} x_1=0\\ x_2=0\\ x_3=0\\ x_4=0. \end{array} \right. \]
Then
\[ \Gamma^{\hat\sigma} \colon \left\{ \begin{array}{ll} x_1=0 \\ h^{q^2-q}x_2+h^{-q-1}x_3+x_5+x_0=0 \end{array} \right. \,\,\, \text{and} \,\,\, \Gamma^{\hat\sigma^2} \colon \left\{ \begin{array}{ll} x_2=0 \\ -h^{-1-q^2}x_3+h^{-q^2-q}x_4+x_0+x_1=0. \end{array} \right. \]
Therefore,
\[ \Gamma\cap\Gamma^{\hat\sigma} \colon \left\{ \begin{array}{llll} x_0=0\\ x_1=0 \\ -h^{q^2-1}x_2+x_4+x_5=0 \\ h^{q^2-q}x_2+h^{-q-1}x_3+x_5=0 \end{array} \right. \,\,\, \text{and} \quad
\Gamma\cap\Gamma^{\hat\sigma}\cap\Gamma^{{\hat\sigma}^2} \colon \left\{ \begin{array}{llllll} x_0=0\\ x_1=0 \\ x_2=0 \\ x_4+x_5=0 \\ h^{-q-1}x_3+x_5=0 \\ -h^{-q^2-1}x_3+h^{-q^2-q}x_4=0. \end{array} \right. \]
Hence, $\dim_{\F_{q^6}} (\Gamma\cap\Gamma^{\hat\sigma})=1$ and
$\dim_{\F_{q^6}} (\Gamma\cap\Gamma^{\hat\sigma}\cap\Gamma^{{\hat\sigma}^2})=-1 $, since $q$ is odd and $h^{q^3+1}\neq 1$.
So, $\mathrm{intn}_{\sigma}(\Gamma)=3$ and hence, by Result \ref{nopseudonoLP} it follows that $\mathcal{L}_h$ is not equivalent neither to $L^1$ nor to $L^{2}_{\delta}$.

Generalizing \cite[Propositions 5.4 and 5.5]{ZZ} we have the following two propositions.

\begin{proposition}\label{L3}
The linear set $\mathcal{L}_h$ is not $\mathrm{P}\Gamma\mathrm{L}$-equivalent to $L^3_{\delta}$.
\end{proposition}
\begin{proof}
By Lemma \ref{equiv}, we have to check whether $\mathcal{U}_h$ and $U^3_{\delta}$ are $\Gamma\mathrm{L}$-equivalent, with $\N_{q^6/q^{3}}(\delta) \notin \{0,1\}$.
Suppose that there exist $\rho \in \mathrm{Aut}(\F_{q^6})$ and an invertible matrix
$\left( \begin{array}{llrr} a & b \\ c & d \end{array} \right)$ such that for each $x \in \F_{q^6}$ there exists $z \in \F_{q^6}$ satisfying  \[ \left(
\begin{array}{llrr}
a & b\\
c & d
\end{array} \right)
\left(
\begin{array}{ccc}
\hspace{1cm}x^\rho\\
h^{\rho(q-1)}x^{\rho q}-h^{\rho(q^2-1)}x^{\rho q^2}+x^{\rho q^4}+x^{\rho q^5}
\end{array} \right)
=\begin{pmatrix} z\\
{ z^q+\delta z^{q^4}} \end{pmatrix}.\]
Equivalently, for each $x \in \F_{q^6}$ we have\footnote{We may replace $h^\rho$ by $h$, since  $h^{q^3+1}=-1$ if and only if $(h^{\rho})^{q^3+1}=-1$.}
\[cx^\rho+d(h^{q-1}x^{\rho q}-h^{q^2-1}x^{\rho q^2}+x^{\rho q^4}+x^{\rho q^5})=a^qx^{\rho q}+\]
\[ +b^q(h^{q^2-q}x^{\rho q^2}+h^{-q-1}x^{\rho q^3}+x^{\rho q^5}+x^\rho)+\delta[a^{q^4}x^{\rho q^4}+b^{q^4}(h^{-q^2+q}x^{\rho q^5}-h^{q+1}x^\rho+x^{\rho q^2}+x^{\rho q^3})]. \]
This is a polynomial identity in $x^{\rho}$ and hence we have the following relations:
\begin{equation}\label{binZ}
\left\{
\begin{array}{llllll}
c=b^q+ \delta h^{q+1} b^{q^4}\\
dh^{q-1}= a^q\\
-dh^{q^2-1}=h^{q^2-q} b^q+\delta b^{q^4}\\
0=h^{-1-q} b^q+\delta b^{q^4}\\
d=\delta a^{q^4}\\
d= b^q+ \delta h^{q-q^2} b^{q^4}.
\end{array}\right.
\end{equation}

From the second and the fifth equations, if $a \neq 0$ then $\delta h^{q-1}=a^{q-q^4}$ and so $\N_{q^6/q^3}(\delta)=1$, which is not possible and so $a=d=0$ and $b,c \neq 0$. By the last equation, we would get $\N_{q^6/q^3}(\delta)=1$, a contradiction.
\end{proof}

\begin{proposition}\label{trin0}
The linear set $\mathcal{L}_h$ is $\mathrm{P}\Gamma\mathrm{L}$-equivalent to $L^4_{\delta}$
(with $\delta^2+\delta=1$) if and only if there exist $a,b,c,d \in \F_{q^6}$
and $\rho \in \mathrm{Aut}(\F_{q^6})$ such that $ad-bc \neq 0$ and either
\begin{equation}\label{trin}
\left\{
\begin{array}{llllll}
c=b^q-\delta k^{q^2+1} b^{q^5}\\
a=-k^{q+1}b^{q^4}-\delta^q b^{q^2}\\
d=k^{-q+1}b^{q^3}+\delta b^{q^5}\\
b^{q^3}+(k^{q-1}+\delta k^{q+q^2})b^{q^5}=0\\
k^{q^2-q}b^q +(1+k^{q^2-q})b^{q^3}+\delta k^{q^2-1}b^{q^5}=0\\
-\delta b^q+(k^{-q+1}+\delta^2 k^{1-q^2})b^{q^3}+\delta b^{q^5}=0
\end{array}\right.
\end{equation}
or
\begin{equation}\label{trin2}
\left\{
\begin{array}{llllll}
c=\delta b^q- k^{q^2+1} b^{q^5}\\
a=-\delta^q k^{q+1}b^{q^4}-b^{q^2}\\
d= k^{-q+1}b^{q^3}+ b^{q^5}\\
\delta b^{q^3}+(k^{q-1}-\delta k^{q^2+q})b^{q^5}=0\\
\delta k^{q^2-q}b^q+(k^{q^2-q}+1)b^{q^3}+k^{q^2-1}b^{q^5}=0\\
\delta^2b^q+(k^{-q+1}+\delta^2k^{-q^2+1})b^{q^3}+b^{q^5}=0,
\end{array}\right.
\end{equation}
where $k=h^\rho$.
\end{proposition}
\begin{proof}
By Lemma \ref{equiv} we have to check whether $\mathcal{U}_h$ is equivalent either to $U^4_{\delta}$ or to $(U^4_{\delta})^\perp$.
\noindent Suppose that there exist $\rho \in \mathrm{Aut}(\F_{q^6})$ and an invertible matrix
$\left( \begin{array}{llrr} a & b \\ c & d \end{array} \right)$ such that for each $x \in \F_{q^6}$ there exists $z \in \F_{q^6}$ satisfying  \[ \left(
\begin{array}{llrr}
a & b\\
c & d
\end{array} \right)
\left(
\begin{array}{ccr}
x^\rho\\
h^{\rho(q-1)}x^{\rho q}-h^{\rho(q^2-1)}x^{\rho q^2}+x^{\rho q^4}+x^{\rho q^5}
\end{array} \right)
=\left(
\begin{array}{ccr}
z\\
z^q+z^{q^3}+\delta z^{q^5} \end{array}
\right).\]
Equivalently, for each $x \in \F_{q^6}$ we have
\[ cx^\rho+d(k^{q-1}x^{\rho q}-k^{q^2-1}x^{\rho q^2}+x^{\rho q^4}+x^{\rho q^5})=a^q x^{\rho q}+b^q(k^{q^2-q}x^{\rho q^2}+k^{-1-q}x^{\rho q^3}+x^{\rho q^5}+x^{\rho})+ \]
\[ +a^{q^3} x^{\rho q^3}+b^{q^3}(k^{-q+1}x^{\rho q^4}-k^{-q^2+1}x^{\rho q^5}+x^{\rho q}+x^{\rho q^2})+ \]
\[ +\delta[a^{q^5}x^{\rho q^5}+b^{q^5}(-k^{1+q^2}x^{\rho}+k^{q^2+q}x^{\rho q}+x^{\rho q^3}+x^{\rho q^4})]. \]
This is a polynomial identity in $x^{\rho}$ which yields to the following equations
\[
\left\{
\begin{array}{llllll}
c=b^q-\delta k^{q^2+1} b^{q^5}\\
dk^{q-1}=a^q+b^{q^3}+\delta k^{q+q^2} b^{q^5}\\
-d k^{q^2-1}=k^{q^2-q}b^q+b^{q^3}\\
0= k^{-q-1}b^q+a^{q^3}+\delta b^{q^5}\\
d= k^{-q+1}b^{q^3}+\delta b^{q^5}\\
d=b^q-k^{-q^2+1}b^{q^3}+\delta a^{q^5}
\end{array}\right.
\]
which can be written as \eqref{trin}.

Now, suppose that there exist $\rho \in \mathrm{Aut}(\F_{q^6})$ and an invertible matrix
$\left( \begin{array}{llrr} a & b \\ c & d \end{array} \right)$ such that for each $x \in \F_{q^6}$ there exists $z \in \F_{q^6}$ satisfying  \[ \left(
\begin{array}{llrr}
a & b\\
c & d
\end{array} \right)
\left(
\begin{array}{ccr}
\hspace{1cm}x^\rho\\
h^{\rho(q-1)}x^{\rho q}-h^{\rho(q^2-1)}x^{\rho q^2}+x^{\rho q^4}+x^{\rho q^5}
\end{array} \right)
=\left(
\begin{array}{ccr}
z\\
\delta z^q+z^{q^3}+ z^{q^5} \end{array}
\right).\]
Equivalently, for each $x \in \F_{q^6}$ we have
\[ cx^\rho+d(k^{q-1}x^{\rho q}-k^{q^2-1}x^{\rho q^2}+x^{\rho q^4}+x^{\rho q^5})=\delta[a^q x^{\rho q}+b^q(k^{q^2-q}x^{\rho q^2}+k^{-1-q}x^{\rho q^3}+x^{\rho q^5}+x^{\rho})]+ \]
\[ +a^{q^3} x^{\rho q^3}+b^{q^3}(k^{-q+1}x^{\rho q^4}-k^{-q^2+1}x^{\rho q^5}+x^{\rho q}+x^{\rho q^2})+ \]
\[ +a^{q^5}x^{\rho q^5}+b^{q^5}(-k^{1+q^2}x^{\rho}+k^{q^2+q}x^{\rho q}+x^{\rho q^3}+x^{\rho q^4}). \]
This is a polynomial identity in $x^{\rho}$ which yields to the following equations
\[
\left\{
\begin{array}{llllll}
c=\delta b^q- k^{q^2+1} b^{q^5}\\
dk^{q-1}=\delta a^q+b^{q^3}+ k^{q+q^2} b^{q^5}\\
-d k^{q^2-1}=\delta k^{q^2-q}b^q+b^{q^3}\\
0= \delta k^{-q-1}b^q+a^{q^3}+ b^{q^5}\\
d= k^{-q+1}b^{q^3}+ b^{q^5}\\
d=\delta b^q-k^{-q^2+1}b^{q^3}+ a^{q^5}
\end{array}\right.
\]
which can be written as \eqref{trin2}.
\end{proof}

We are now ready to prove that when $h\notin \F_{q^2}$, $\mathcal{L}_h$ is new.

\begin{proposition}\label{L4hnoFq2}
If $h\notin \F_{q^2}$, then $\mathcal{L}_h$ is not $\mathrm{P}\Gamma\mathrm{L}$-equivalent to $L^4_{\delta}$ (with $\delta^2+\delta=1$).
\end{proposition}
\begin{proof}
By Proposition \ref{trin0} we have to show that there are no $a,b,c$ and $d$ in $\F_{q^6}$ such that $ad-bc\neq 0$ and \eqref{trin} or \eqref{trin2} are satisfied.
Note that $b=0$ in \eqref{trin} and \eqref{trin2}  yields $a=c=d=0$,  a contradiction. So, suppose $b\neq0$.
Since $h \notin \F_{q^2}$ then $k \notin \F_{q^2}$.
We start by proving that the last three equations of \eqref{trin}, i.e.
\[ \left\{ \begin{array}{lll}
\mathrm{Eq}_1:b^{q^3}+(k^{q-1}+\delta k^{q+q^2})b^{q^5}=0\\
\mathrm{Eq}_2:k^{q^2-q}b^q +(1+k^{q^2-q})b^{q^3}+\delta k^{q^2-1}b^{q^5}=0\\
\mathrm{Eq}_3:-\delta b^q+(k^{-q+1}+\delta^2 k^{1-q^2})b^{q^3}+\delta b^{q^5}=0 \end{array} \right., \]
yield a contradiction.
As in the above section, we will consider the $q$-th powers of $\mathrm{Eq}_1$, $\mathrm{Eq}_2$ and $\mathrm{Eq}_3$ replacing $b^{q^i}$, $k^{q^j}$, and $\delta^{q^\ell}$ (respectively) by $X_i$, $Y_j$, and $Z_\ell$ with $i,j \in \{0,1,2,3,4,5\}$ and $\ell \in \{0,1\}$.
Consider the set $S$ of polynomials in the variables $X_i,Y_j$, and $Z_\ell$
\[S:=\{ \mathrm{Eq}_1^{q^\alpha}, \mathrm{Eq}_2^{q^\beta}, \mathrm{Eq}_3^{q^\gamma} \colon \alpha, \beta, \gamma \in \{0,1,2,3,4,5\} \}.\]
By eliminating from $S$ the variables $X_5$, $X_4$, $X_3$, and $X_2$ using $\mathrm{Eq}_1$, $\mathrm{Eq}_1^q$, $\mathrm{Eq}_1^{q^4}$, and $\mathrm{Eq}_1^{q^3}$ respectively  we obtain
\[ X_0Y_1(Z_1Y_0^2Y_2-Z_1Y_0Y_2^2-Z_1Y_0+Z_1Y_2-Z_0^2Z_2-Z_2)=0. \]
By the conditions on $b$ and $k$, $X_0Y_1\neq 0$ and therefore
\[P:=Z_1Y_0^2Y_2-Z_1Y_0Y_2^2-Z_1Y_0+Z_1Y_2-Z_0^2Z_2-Z_2=0.\]
We eliminate $Z_1$ in $S$ using $P$, obtaining, w.r.t. $b$, $k$, and $\delta$,
\[ bk^{q^2+1}(k-k^q)(k+k^q)(k^{q^2+1}-1)(k^{q^2+1}+1)=0,\]
 a contradiction to $k \notin\F_{q^2}$.

\smallskip

Consider now the last three equations of \eqref{trin2}, i.e.
\[ \left\{ \begin{array}{lll}
\mathrm{Eq}_1:\delta b^{q^3}+(k^{q-1}-\delta k^{q^2+q})b^{q^5}=0\\
\mathrm{Eq}_2:\delta k^{q^2-q}b^q+(k^{q^2-q}+1)b^{q^3}+k^{q^2-1}b^{q^5}=0\\
\mathrm{Eq}_3:\delta^2b^q+(k^{-q+1}+\delta^2k^{-q^2+1})b^{q^3}+b^{q^5}=0\\ \end{array} \right.. \]
As before, we will consider the $q$-th powers of $\mathrm{Eq}_1$, $\mathrm{Eq}_2$, and $\mathrm{Eq}_3$ replacing $b^{q^i}$, $k^{q^j}$, and $\delta^{q^\ell}$ (respectively) by $X_i$, $Y_j$, and $Z_\ell$ with $i,j \in \{0,1,2,3,4,5\}$ and $\ell \in \{0,1\}$.
Consider the set $S$ of polynomials in the variables $X_i,Y_j$ and $Z_\ell$
\[S:=\{ \mathrm{Eq}_1^{q^\alpha}, \mathrm{Eq}_2^{q^\beta}, \mathrm{Eq}_3^{q^\gamma} \colon \alpha, \beta, \gamma \in \{0,1,2,3,4,5\} \}.\]
We eliminate in $S$ the variables $X_5$, $X_4$, $X_3$,  and $X_2$ using  $\mathrm{Eq}_1$, $\mathrm{Eq}_1^q$, $\mathrm{Eq}_1^{q^4}$, and $\mathrm{Eq}_1^{q^3}$ respectively, and we get
\[ Y_0X_0(Z_1Y_0^2Y_2^2+2Z_1Y_0Y_1^2Y_2+2Z_1Y_0Y_2+Z_1Y_1^2-Y_0^2Y_2^2-Y_0Y_1^2Y_2-Y_0Y_2-Y_1^2)=0. \]
Since $b\neq 0$ and $k \notin\F_{q^2}$, $X_0Y_0\neq 0$ and therefore
\[P:=Z_1Y_0^2Y_2^2+2Z_1Y_0Y_1^2Y_2+2Z_1Y_0Y_2+Z_1Y_1^2-Y_0^2Y_2^2-Y_0Y_1^2Y_2-Y_0Y_2-Y_1^2=0.\]
Once again we  consider the resultants of the polynomials in $S$ and $P$ w.r.t. $Z_1$ and we obtain
\[bk^{q^2+2q}(k-k^q)(k+k^q)(k^{q^2+1}-1)(k^{q^2+1}+1)=0,\]
a contradiction to $k \notin\F_{q^2}$.
\end{proof}

As a consequence of the above considerations and Propositions \ref{L3} and \ref{L4hnoFq2}, we have the following.

\begin{corollary}\label{nonequiv}
If $h\notin \F_{q^2}$, then $\mathcal{L}_h$ is not $\mathrm{P}\Gamma\mathrm{L}$-equivalent to any known scattered linear set in $\PG(1,q^6)$.
\end{corollary}

\subsection{$\mathcal{L}_h$ may be defined by a trinomial}

Suppose that $h \in \F_{q^2}$, then the condition on $h$ becomes $h^{q+1}=-1$.
For such $h$ we can prove that the linear set $\mathcal{L}_h$ can be defined by the $q$-polynomial $(h^{-1}-1)x^q+x^{q^3}+(h-1)x^{q^5}$.

\begin{proposition}\label{ltri}
If $h\in \F_{q^2}$, then the linear set $\mathcal{L}_h$ is $\mathrm{P}\Gamma\mathrm{L}$-equivalent to
\[ L_{\mathrm{tri}}:=\{ \la (x,(h^{-1}-1)x^q+x^{q^3}+(h-1)x^{q^5}) \ra_{\F_{q^6}} \colon x \in \F_{q^6}^* \}. \]
\end{proposition}
\begin{proof}
Let $A=\left( \begin{array}{llrr} a & b \\ c & d \end{array} \right)\in \mathrm{GL}(2,q^6)$ with $a=-h+h^{-1}, b=1, c=h^{-1}-1-h^3+h^2$ and $d=h-h^2-1$. Straightforward computations show that the subspaces $\mathcal{U}_h$ and $U_{(h^{-1}-1)x^q+x^{q^3}+(h-1)x^{q^5}}$ are $\Gamma\mathrm{L}(2,q^6)$-equivalent under the action of the matrix $A$.
Hence, the linear sets $\mathcal{L}_h$ and $L_{\mathrm{tri}}$ are  $\mathrm{P}\Gamma\mathrm{L}$-equivalent.
\end{proof}

The fact that $\mathcal{L}_h$ can also be defined by a trinomial will help us to completely close the equivalence issue for $\mathcal{L}_h$ when $h \in \F_{q^2}$.
Indeed, we can prove the following:

\begin{proposition}\label{eqL4dtrin}
If $h\in\F_{q^2}$, then the linear set $\mathcal L_h$ is  $\mathrm{P}\Gamma\mathrm{L}$-equivalent to
some $L_{\delta}^4$ ($\delta^2+\delta=1$) if and only if $h\in\Fq$ and $q$ is a power of $5$.
\end{proposition}
\begin{proof}
Recall that by \cite[Proposition 5.5]{ZZ} if $h \in \F_q$ and $q$ is a power of $5$, then  $\mathcal L_h$ is $\mathrm{P}\Gamma\mathrm{L}$-equivalent to some $L_{\delta}^4$.
As in the proof of Proposition \ref{trin0}, by Lemma \ref{equiv} we have to check whether $U_{(h^{-1}-1)x^q+x^{q^3}+(h-1)x^{q^5}}$ is
$\Gamma\mathrm{L}$-equivalent either to $U^4_{\delta}$ or to $(U^4_{\delta})^\perp$.
\noindent Suppose that there exist $\rho \in \mathrm{Aut}(\F_{q^6})$ and an invertible matrix
$\left( \begin{array}{llrr} a & b \\ c & d \end{array} \right)$ such that for each $x \in \F_{q^6}$ there exists $z \in \F_{q^6}$ satisfying  \[ \left(
\begin{array}{llrr}
a & b\\
c & d
\end{array} \right)
\left(
\begin{array}{cc}
x^\rho\\
(h^{-\rho}-1)x^{\rho q}+x^{\rho q^3}+(h^\rho-1)x^{\rho q^5}
\end{array} \right)
=\left(
\begin{array}{cc}
z\\
z^q+z^{q^3}+\delta z^{q^5} \end{array}
\right).\]
Let $k=h^\rho$, for which $k^{q+1}=-1$.
As in Proposition \ref{L3}, we obtain a polynomial identity, whence
\begin{equation}\label{trinomio}
\left\{ \begin{array}{llllll}
c=b^q(k^q-1)+b^{q^3}+\delta b^{q^5}(k^{-q}-1)\\
d(k^{-1}-1)=a^q\\
0=b^q(k^{-q}-1)+b^{q^3}(k^q-1)+b^{q^5}\delta\\
d=a^{q^3}\\
0=b^q+b^{q^3}(k^{-q}-1)+b^{q^5}(k^q-1)\delta\\
d(k-1)=\delta a^{q^5}.
\end{array}
\right.
\end{equation}
By subtracting the fifth equation from the third equation raised to $q^2$, we get
\[ b^q=b^{q^5}(k^q-1), \]
i.e. either $b=0$ or $k^q-1=(b^q)^{q^4-1}$, whence
we get either $b=0$ or $\N_{q^6/q^2}(k^q-1)=1$.

If $b\neq 0$, since $k-1 \in \F_{q^2}$ and $\N_{q^6/q^2}(k-1)=(k-1)^3=1$, then
\[ k^3-3k^2+3k-2=0 \]
and, since $\N_{q^6/q^2}(k^q-1)=1$ and $k^q=-1/k$,
\[ 2k^3+3k^2+3k+1=0, \]
from which we get
\begin{equation}\label{eq2grado}
9k^2-3k+5=0.
\end{equation}
\begin{itemize}
\item If $k \notin \F_q$ then $k$ and $k^q$ are the solutions of \eqref{eq2grado} and
\[-1=k^{q+1}=\frac{5}9,\]
which holds if and only if $q$ is a power of $7$. By \eqref{eq2grado} it follows that $k \in \F_q$, a contradiction.
\item If $k \in \F_q$, then $k^2=-1$ and by \eqref{eq2grado} we have $k=-4/3$, which is possible if and only if $q$ is a power of $5$.
\end{itemize}
Hence, if either $k\notin \F_q$ or $k\in \F_q$ with $q$ not a power of $5$, we have that $b=0$ and hence $c=0$, $a\neq0$ and $d\neq0$.

By combining the second and the fourth equation of \eqref{trinomio}, we get $\N_{q^6/q^2}(k^{-1}-1)=1$ and, since $k^q=-1/k$, $\N_{q^6/q^2}(k^{q}+1)=-1$.
Arguing as above, we get a contradiction whenever $k\notin\F_q$ or $k\in\F_q$ with $q$ not a power of $5$.

Now, suppose that there exist $\rho \in \mathrm{Aut}(\F_{q^6})$ and an invertible matrix
$\left( \begin{array}{llrr} a & b \\ c & d \end{array} \right)$ such that for each $x \in \F_{q^6}$ there exists $z \in \F_{q^6}$ satisfying  \[ \left(
\begin{array}{llrr}
a & b\\
c & d
\end{array} \right)
\left(
\begin{array}{cc}
x^\rho\\
(h^{-\rho}-1)x^{\rho q}+x^{\rho q^3}+(h^\rho-1)x^{\rho q^5}
\end{array} \right)
=\left(
\begin{array}{cc}
z\\
\delta z^q+z^{q^3}+ z^{q^5} \end{array}
\right).\]
Let $k=h^\rho$.
As before, we get the following equations
\begin{equation}\label{trinomio2}
\left\{ \begin{array}{llllll}
c=\delta b^q(k^q-1)+b^{q^3}+b^{q^5}(k^{-q}-1)\\
d(k^{-1}-1)=\delta a^q\\
0=\delta b^q(k^{-q}-1)+b^{q^3}(k^q-1)+b^{q^5}\\
d=a^{q^3}\\
0=\delta b^q+b^{q^3}(k^{-q}-1)+b^{q^5}(k^q-1)\\
d(k-1)= a^{q^5}.
\end{array}
\right.
\end{equation}
By subtracting the fifth equation from the third raised to $q^2$ of the above system we get
\[ b^q=b^{q^3}(k^{-q}-1). \]
If $b\neq0$, then $\N_{q^6/q^2}(k^{-q}-1)=1$.
Hence, arguing as above, we get that $b=0$ and hence $c=0$, $a,d \neq 0$.
By combining the  fourth equation with the second and the fifth equation of \eqref{trinomio2} we get $\N_{q^6/q^2}(k-1)=1$, which yields again to a contradiction when $k \notin \F_q$ or $k \in \F_q$ with $q$ not a power of $5$.
\end{proof}

So, as a consequence of Corollary \ref{nonequiv} and of the above proposition, we have the following result.

\begin{corollary}\label{EquivLS}
Apart from the case $h \in \F_q$ and $q$ a power of $5$, the linear set $\mathcal{L}_h$ is not $\mathrm{P}\Gamma\mathrm{L}$-equivalent to any known scattered linear set in $\PG(1,q^6)$.
\end{corollary}

By Proposition \ref{ltri}, when $h \in \F_{q^2}$, $\mathcal{L}_h$ is a linear set of the family presented in \cite[Section 7]{PZ}.
Also, we get an extension of \cite[Table 1]{MMZ}, where it is shown examples of scattered linear sets which could generalize the family presented in \cite{CsMZ2018}.
We do not know whether the linear set $\mathcal{L}_h$, for each $h \in\mathbb{F}_{q^6}\setminus\mathbb{F}_{q^2}$ with $h^{q^3+1}=-1$, may be defined by a trinomial or not.

\section{New MRD-codes}\label{Sec:MRD}

Delsarte in \cite{Delsarte} (see also \cite{Gabidulin}) introduced in 1978 rank metric codes as follows.
A  \emph{rank metric code} (or \emph{RM}-code for short) $\mathcal C$ is a subset of the
set of $m \times n$ matrices $\F_q^{m\times n}$ over $\F_q$ equipped with the distance function
\[d(A,B) = \mathrm{rk}\,(A-B)\]
for $A,B \in \F_q^{m\times n}$.
The \emph{minimum distance} of $\C$ is
\[d = \min\{ d(A,B) \colon A,B \in \C,\,\, A\neq B \}.\]
We will say that a rank metric code of $\F_q^{m\times n}$ with minimum distance $d$ has parameters $(m,n,q;d)$.
When $\C$ is an $\F_q$-subspace of $\F_q^{m\times n}$, we say that $\C$ is $\F_q$-linear.
In the same paper, Delsarte also showed that the parameters of these codes fulfill a Singleton-like bound, i.e.
\[ |\C| \leq q^{\max\{m,n\}(\min\{m,n\}-d+1)}. \]
When the equality holds, we call $\C$ a \emph{maximum rank distance} (\emph{MRD} for short) code.
We will consider only the case $m=n$ and we will use the following equivalence definition for codes of $\F_q^{m \times m}$.
Two $\F_q$-linear RM-codes $\C$ and $\C'$ are
\emph{equivalent} if and only if there exist two invertible matrices $A,B \in \F_q^{m \times m}$ and a field automorphism $\sigma$ such that $\{A C^\sigma B \colon C\in \C\}=\C'$, or $\{A C^{T\sigma}B \colon C\in \C\}=\C'$, where $T$ denotes transposition.
Also, the \emph{left} and \emph{right idealisers} of $\C$ are $L(\C)=\{A \in \mathrm{GL}(m,q) \colon A \C\subseteq \C\}$ and $R(\C)=\{B \in \mathrm{GL}(m,q) \colon \C B \subseteq \C\}$, \cite{LN2016,LTZ2}. They are important invariants for linear rank metric codes, see also \cite{GiuZ} for further invariants.

\medskip

In \cite[Section 5]{Sh} Sheekey showed that scattered $\F_q$-linear sets of $\PG(1,q^n)$ of rank $n$ yield $\F_q$-linear MRD-codes with parameters $(n,n,q;n-1)$ with left idealiser isomorphic to $\F_{q^n}$; see \cite{CSMPZ2016,CsMPZ2019,ShVdV} for further details on such kind of connections.
We briefly recall here the construction from \cite{Sh}. Let $U_f=\{(x,f(x))\colon x\in \F_{q^n}\}$
for some scattered $q$-polynomial $f(x)$.
After fixing an $\F_q$-basis for $\F_{q^n}$ we can define an isomorphism between the rings $\mathrm{End}(\F_{q^n},\F_q)$ and $\F_q^{n\times n}$.
In this way the set
\[
\C_f:=\{x\mapsto af(x)+bx \colon a,b \in \F_{q^n}\}
\]
corresponds to a set of $n\times n$ matrices over $\F_q$ forming an $\F_q$-linear MRD-code with parameters $(n,n,q;n-1)$. Also, since $\C_f$ is an $\F_{q^n}$-subspace of $\mathrm{End}(\F_{q^n},\F_q)$
its left idealiser $L(\C_f)$ is isomorphic to $\F_{q^n}$.
For further details see \cite[Section 6]{CMPZ}.

Let $\C_f$ and $\C_h$ be two MRD-codes arising from maximum scattered subspaces $U_f$ and $U_h$ of $\F_{q^n}\times \F_{q^n}$.
In \cite[Theorem 8]{Sh} the author showed that there exist invertible matrices $A$, $B$ and $\sigma \in \mathrm{Aut}(\F_{q})$ such that $A \C_f^\sigma B=\C_h$ if and only if $U_f$ and $U_h$ are  $\Gamma\mathrm{L}(2,q^n)$-equivalent

Therefore, we have the following.

\begin{theorem}
\label{thm:newMRD}
The $\F_q$-linear MRD-code $\C_{f_h}$ arising from the $\F_q$-subspace $\mathcal{U}_h$ has parameters $(6,6,q;5)$ and left idealiser isomorphic to $\F_{q^6}$,
and is not equivalent to any previously known MRD-code, apart from the case $h \in \F_q$ and $q$ a power of $5$.
\end{theorem}
\begin{proof}
From \cite[Section 6]{CMPZ}, the previously known $\F_q$-linear MRD-codes with parameters $(6,6,q;5)$ and with left idealiser isomorphic to $\F_{q^6}$ arise, up to equivalence, from one of the
maximum scattered subspaces of $\F_{q^{6}}\times\F_{q^{6}}$ described in Section \ref{EquivIssue}. From Corollaries \ref{nonequiv} and \ref{EquivLS} the result then follows.
\end{proof}

\end{document}